\documentclass[twoside,a4paper,reqno,11pt]{amsart} 
\usepackage[top=30mm,right=30mm,bottom=30mm,left=30mm]{geometry}
\usepackage{amsmath, amssymb, amsthm, amsfonts}
\usepackage{graphicx, float, array, multicol, rotating}
\usepackage{mathtools,mathdots,wasysym,latexsym,stmaryrd,gensymb}

\usepackage{color}

\renewcommand{\leq}{\leqslant}
\newcommand{\leqn}{\trianglelefteqslant}
\renewcommand{\geq}{\geqslant}

\newcommand{\sym}{\mathrm{Sym}}
\newcommand{\alt}{\mathrm{Alt}}
\newcommand{\GL}{\mathrm{GL}}
\newcommand{\SL}{\mathrm{SL}}
\newcommand{\GammaL}{\Gamma\mathrm{L}}
\newcommand{\PSL}{\mathrm{PSL}}
\newcommand{\PGL}{\mathrm{PGL}}
\newcommand{\PGammaL}{\mathrm{P}\Gamma\mathrm{L}}
\newcommand{\nonsplit}[2]{#1\raisebox{0.6ex}{$\cdot$} #2}

\usepackage[shortlabels]{enumitem}
\setlist[enumerate]{label=\rm{(\roman*)}}

\usepackage[justification=centering,tableposition=top,figureposition=bottom]{caption}
\usepackage{ltablex,hhline,booktabs}

\theoremstyle{definition}
\newtheorem{definition}{Definition}[section]
\newtheorem{shdefinition}{Definition}

\theoremstyle{plain}

\newtheorem{shtheorem}[shdefinition]{Theorem}
\newtheorem{lemma}[definition]{Lemma}

\newtheorem{shcorollary}[shdefinition]{Corollary}

\begin{document}
\title[The distinguishing number of quasiprimitive and semiprimitive groups]{The distinguishing number of quasiprimitive \\ and semiprimitive groups}

\author{Alice Devillers, Luke Morgan, Scott Harper}

\address{Alice Devillers, Centre for the Mathematics of Symmetry and Computation, Department of Mathematics and Statistics (M019), The University of Western Australia, 35 Stirling Highway,  Crawley, 6009, Western Australia}
\email{alice.devillers@uwa.edu.au}

\address{Scott Harper, School of Mathematics, University of Bristol, Bristol, BS8 1TW, UK}
\email{scott.harper@bristol.ac.uk}
 
 \address{Luke Morgan,   University of Primorska, UP IAM, Muzejski trg 2,  6000 Koper, Slovenia.
University of Primorska, UP FAMNIT, Glagolja\v{s}ka 8,  6000 Koper,    }
 \email{luke.morgan@famnit.upr.si}

\thanks{\emph{Acknowledgements}: The authors thank Gabriel Verret for highlighting this problem and the Centre for the Mathematics of Symmetry and Computation, where this work began at the 2018 Research Retreat. The second author is grateful for the Cecil King Travel Scholarship from the LMS and the hospitality of the University of Western Australia; he also thanks EPSRC and the Heilbronn Institute for Mathematical Research for their financial support. The third author gratefully acknowledges the support of the ARC grant DE160100081.}
 

\date{\today}

\begin{abstract}
The distinguishing number of  $G \leq \sym(\Omega)$ is the smallest size of a partition of $\Omega$ such that only the identity of $G$ fixes all the parts of the partition. Extending earlier results of Cameron, Neumann, Saxl and Seress on the distinguishing number of finite primitive  groups, we show that all imprimitive quasiprimitive groups have distinguishing number two, and all non-quasiprimitive semiprimitive groups have distinguishing number two, except for $\mathrm{GL}(2,3)$ acting on the eight non-zero vectors of $\mathbb{F}_3^2$, which has distinguishing number three.
\end{abstract}

\maketitle

\section{Introduction}\label{s:intro}

Let $\Omega$ be a finite set and let $\sym(\Omega)$ be the symmetric and alternating groups on $\Omega$. The \emph{distinguishing number} of a permutation group $G \leq \sym(\Omega)$, denoted $D(G)$, is the least positive integer $k$ such that there exists a partition $\Pi$ of $\Omega$ into $k$ parts such that only the identity of $G$ stabilises each part.

It is straightforward to see that $D(G) = 1$ if and only if $G = 1$. In addition, $D(G)=2$ if and only if there is a partition $\Pi$ of $\Omega$ consisting of a subset $\pi \subseteq \Omega$ and its complement $\overline{\pi}$ such that $1 = G_{(\Pi)}= G_{ \pi } \cap G_{ \overline{\pi} } = G_{ \pi }$. Therefore, $D(G)=2$ is equivalent to $G$ having a regular orbit in its action on the power set of $\Omega$. For this reason, it is of particular interest to determine whether a group has a distinguishing number of two.

Along with the \emph{base size} (denoted $b(G)$, the smallest size of a subset of $\Omega$ such that only the identity of $G$ fixes the subset pointwise), the distinguishing number is a well-studied invariant of permutation groups. In fact, it is not difficult to see that these two invariants satisfy the inequality $D(G) \leq b(G) + 1$. We recommend the article of Bailey and Cameron \cite{bailey_cameron}, which surveys many of the important results on this topic and highlights further connections between the distinguishing number, base size and other interesting invariants of groups and graphs.

Primitive permutation groups are the most widely studied class of permutation groups and they are classified by the O'Nan--Scott Theorem (see \cite{LPSonanscott}). It is shown in \cite{cameron,cameron_neumann_saxl} that, apart from the alternating and symmetric groups in their natural actions, all but finitely many primitive permutation groups have distinguishing number two. Seress determined the list of such exceptions and, thus, established the following result \cite[Theorem~2]{seress}. (Throughout, if $|\Omega|=n$, then we will write $\sym(n)=\sym(\Omega)$ and $\alt(n)=\alt(\Omega)$.)

\begin{shtheorem}\label{t:seress}
Let $G \leq \sym(n)$ be primitive. Then one of the following holds:
\begin{enumerate}
\item $G=\sym(n)$ and $D(G)=n$;
\item $G=\alt(n)$ and $D(G)=n-1$;
\item $G \in \mathcal{P}$, where $\mathcal{P}$ is a known list of 43 permutation groups of degree at most 32;
\item $D(G) = 2$.
\end{enumerate}
\end{shtheorem}

An explicit description of the set $\mathcal{P}$ is given in \cite[Theorem~2]{seress} and we also give details in Section~\ref{s:computation}. As a consequence of \cite[Lemma~1]{dolfi}, the distinguishing number of each permutation group in $\mathcal{P}$ is known. In particular, $D(G) \leq 4$ for all $G \in \mathcal{P}$.
 
It is striking to us that the distinguishing number of all primitive groups (except the alternating and symmetric groups) is bounded by an absolute constant. In contrast, the base size of primitive   groups is \emph{not} bounded by an absolute constant (witnessed, for example, by the action of the symmetric group on subsets, see \cite{halasi}). The goal of this paper is to show that the dichotomy exhibited in Theorem~\ref{t:seress} by primitive groups is also  respected by a much wider class of permutation groups.

The most natural generalisation of the class of primitive groups is the class of \emph{quasiprimitive} permutation groups. These groups have been widely studied because of their utility in investigations into the automorphism groups of highly transitive graphs (see \cite{praeger_quasiprim}). A permutation group is quasiprimitive if each non-trivial normal subgroup is transitive. 

Recently, Duyan, Halasi and Mar\'oti completed the proof of Pyber's conjecture \cite{duyan_halasi_maroti}, and an important component of this proof is their result \cite[Theorem~1.2]{duyan_halasi_maroti} that if $G$ is a transitive permutation group of degree $n$, then 
\begin{equation}
|G|^{\frac{1}{n}} \leq D(G) \leq 48\,|G|^{\frac{1}{n}}. \label{e:dhm}
\end{equation}
In proving \eqref{e:dhm}, they establish that $D(G) \leq 4$ for all quasiprimitive groups (other than the alternating or symmetric groups). We improve this bound to obtain the following exact result.
\begin{shtheorem}\label{t:qp}
Let $G$ be a quasiprimitive permutation group that is not primitive. Then $D(G) = 2$.
\end{shtheorem}

The class of \emph{semiprimitive} groups consists of the permutation groups for which each normal subgroup is transitive or semiregular. Evidently, every quasiprimitive group is semiprimitive. The interest in semiprimitive groups comes primarily from investigations into collapsing monoids \cite{bereczky_maroti} and conjectures in algebraic graph theory \cite{potocnik_spiga_verret}. For all semiprimitive groups $G$ of degree $n$ (other than $\alt(n)$ or $\sym(n)$), it was recently proved that $|G| < 4^n$ (see \cite[Theorem 1.5(1)]{kyle}), which together with \eqref{e:dhm} implies that $D(G) < 192$. In fact, we will prove the following much stronger result.

\begin{shtheorem}\label{t:sp}
Let $G$ be a semiprimitive permutation group that is not quasiprimitive. Then $D(G)=2$ or $G=\GL(2,3)$ in its degree $8$ action and $D(G)=3$.
\end{shtheorem}

Theorem~\ref{t:sp} encompasses a wide class of permutation groups and we highlight one specific application. The distinguishing number of $G=\GL(d,q)$ acting on the non-zero vectors of $\mathbb F_q^d$ is considered in \cite{chanI,klavzar_wong_zhu}. Since this group is semiprimitive (and not quasiprimitive unless $q=2$), the results of \cite{chanI,klavzar_wong_zhu} can be recovered with a straightforward application of Theorems~\ref{t:seress}--\ref{t:sp}. We have $D(G)=2$ unless $(d,q) \in \{(2,2), (4,2), (2,3)\}$ and $D(G)=3$.

In general, there is no absolute bound on the distinguishing number of imprimitive groups; for example, the action of $\sym(m) \wr \sym(m)$ on the set $\{ (i,j) \mid 1\leqslant i,j\leqslant m\}$ has distinguishing number $m+1$ (see \cite[Theorem~2.3]{chanII} for a more general statement). It would be interesting to know if anything sensible can be said about the class of permutation groups with distinguishing number bounded by some constant.

When combined with Theorem~\ref{t:seress} and \cite[Lemma~1]{dolfi}, our main results have the following consequence. 
\begin{shcorollary}\label{c:corollary}
Let $G \leq \sym(\Omega)$ be semiprimitive and assume that $\alt(\Omega) \not\leq G$. Then $G$ has a regular orbit on the power set of $\Omega$ (i.e. $D(G)=2$) or $G$ is one of 44 groups. In all cases, $D(G) \leq 4$.
\end{shcorollary}

Interest in the distinguishing number began in graph theory, so we provide one application of our results in this context, which we prove in Section~\ref{s:proof}.

\begin{shcorollary}\label{c:graph}
Let $\Gamma$ be a finite simple connected graph which is not complete and assume that $\mathrm{Aut}(\Gamma)$ is semiprimitive.  Then $D(\Gamma) = 2$, unless $\Gamma$ is the cycle of length $5$, the Hamming graph $H(2,3)$, the Petersen graph or its complement, in which case $D(\Gamma)=3$. 
\end{shcorollary}

Note that the distinguishing number of a (simple undirected) graph $\Gamma$ is the distinguishing number of the permutation group induced on the vertex set. The class of graphs covered by the above corollary is large and well studied. For example, Praeger proved that the automorphism group of every non-bipartite 2-arc-transitive graph is semiprimitive (see \cite[Lemma 1.6]{praeger_imprim}).

This paper is laid out as follows. In Section~\ref{s:reductions} we provide several reduction tools, which allow us to connect the semiprimitive, quasiprimitive and primitive cases. In some cases, this leads to situations that require computational power, details of which are given in Section~\ref{s:computation}. Finally, the main results are proved in Section~\ref{s:proof}.

\section{Reduction lemmas}\label{s:reductions}

In this section, we will present several preliminary results that draw connections between the distinguishing numbers of primitive, quasiprimitive and semiprimitive groups. 

Let $\Omega$ be a finite set. Throughout, if $G \leq \sym(\Omega)$, then by a \emph{maximal $G$-invariant partition of $\Omega$} we mean a $G$-invariant partition $\Sigma$ such that no strictly coarser partition with at least two parts is $G$-invariant (that is to say, there is no $G$-invariant partition $\Sigma'\neq \Sigma$ with at least two parts all of whose parts are unions of parts of $\Sigma$). The group $G^\Sigma$ induced on $\Sigma$  is primitive if and only if $\Sigma$ is maximal partition. Further, if $G$ is quasiprimitive, then $G^\Sigma \cong G$ (as abstract groups), see the final paragraph of \cite[Section 1]{praeger_shalev}.

Let $\omega\in\Omega$. If $L$ is a proper subgroup of $G$ which contains $G_\omega$, then the orbit $\omega^L$ gives rise to  $\Sigma_L:= \{ \{\omega^L \}^g : g\in G\}$, a non-trivial $G$-invariant partition of $\Omega$. Conversely, if $\Sigma$ is a  $G$-invariant partition of $\Omega$, then $\Sigma=\Sigma_L$ where $L=G_\sigma$  for the unique part $\sigma \in \Sigma$ such that $\omega \in \sigma$. Moreover, the action of $G$ on $\Sigma_L$ is isomorphic to the action on the set $G/L$ of right cosets of $L$ in $G$. It is easy to see that $\Sigma_L$ is a maximal $G$-invariant partition of $\Omega$ if and only if $L$ is a maximal subgroup of $G$.

We first record a result on semiprimitive groups that we will make use of throughout the paper.

\begin{lemma}\label{l:sp}
Let $G \leq \sym(\Omega)$ be semiprimitive but not quasiprimitive, let $N \leqn G$ be an intransitive normal subgroup and let $\Sigma$ be the set of $N$-orbits. Then the following hold.
\begin{enumerate}
\item The kernel of the action of $G$ on $\Sigma$ is $N$. In particular, $G^\Sigma \cong G/N$.
\item The group $G^\Sigma$ is semiprimitive.
\item The group $G^\Sigma$ is quasiprimitive if and only if $N$ is maximal among intransitive normal subgroups of $G$.
\item If $\sigma \in \Sigma$ and $\omega \in \sigma$, then $(G^\Sigma)_{\sigma} = G_{\omega}N/N \cong G_{\omega}$. 
\end{enumerate} 
\end{lemma}
\begin{proof}
For (i), (ii) and (iv) see \cite[Lemma 3.1]{lukemichael}. Part (iii) is clear.
\end{proof}

The following  result is crucial. It will allow us to relate the distinguishing number of a semiprimitive (or quasiprimitive) group with the distinguishing number of the group induced on a partition.

\begin{lemma}\label{l:reduction}
Let $G \leq \sym(\Omega)$ be semiprimitive and let $\Sigma$ be a non-trivial $G$-invariant partition of $\Omega$. Then
\begin{enumerate}
\item $D(G) \leq D(G^\Sigma)$;
\item if $|\sigma| \geq |\Sigma|-1$ for all $\sigma \in \Sigma$, then $D(G)=2$.
\end{enumerate}
\end{lemma}

\begin{proof} 
Let $N$ be the kernel of the action of $G$ on $\Sigma$ (so $G^\Sigma \cong G/N$). Since $N$ is normal in $G$ and is intransitive, $N$ must be semiregular.

Let $d = D(G^\Sigma)$ and let $\Pi = \{ \pi_1, \dots, \pi_d \}$ be a partition of the parts of $\Sigma$ such that $(G^\Sigma)_{(\Pi)}=1$. There is an obvious way to extend $\Pi$ to a distinguishing partition of $\Omega$ when $G$ is quasiprimitive (see \cite[Lemma 2.7]{duyan_halasi_maroti}), but we must take more care when $N$ is non-trivial.

Fix a part $\tau \in \pi_1$ and  $\alpha \in \tau$. Let $\Pi'= \{ \pi_1',\dots,\pi_d' \}$ be the partition of $\Omega$, where for each $i$, $\pi'_i$ consists of the points in $\Omega$ that are in a $\Sigma$-part belonging to $\pi_i$, except we remove $\alpha$ from $\pi'_1$ and add it to $\pi'_2$ instead. 

Let $g \in G_{(\Pi')}$. Suppose that $\alpha^g \neq \alpha$. By the definition of $\pi_2'$, we see that $\alpha^g \in \sigma$ for some $\sigma \in \pi_2$. Therefore, $\tau^g=\sigma \in \pi_2$. Since $|\tau| \geq 2$, there exists $\beta \in \tau$ with $\beta \neq \alpha$. Record that $\beta^g \in \sigma$, so $\beta^g \in \pi_2'$. However, $\beta \in \pi_1'$, since $\beta \in \tau $ and $\tau\in \pi_1$, so $\beta^g \in \pi_1'$, which is a contradiction since $\pi_1'$ and $\pi_2'$ are disjoint. Therefore, $\alpha^g=\alpha$.

Since $g$ fixes $\alpha$ and also stabilises $\pi_i'$ setwise, it easily follows that $g$, in its induced action on $\Sigma$, stabilises $\Pi$, and so $g$ is in the kernel $N$ of the action of $G$ on $\Sigma$. Since $N$ is semiregular and $g$ fixes $\alpha$, we conclude that $g=1$. Therefore, $G_{(\Pi')} = 1$, so $D(G) \leq d = D(G^\Sigma)$. This proves part~(i). 

Now consider part~(ii). Let $\sigma_1, \dots, \sigma_n$ be the parts of $\Sigma$, and assume that $|\sigma_i| \geq n-1$. Let $A$ be a set consisting of exactly $i-1$ points from $\sigma_i$ for each $1 \leq i \leq n$. Let $g \in G_A$. Therefore, $g$ stabilises $A$ and permutes the parts of $\Sigma$. Since the $|A \cap \sigma_i| \neq |A \cap \sigma_j|$ for $i \neq j$ we may conclude that $g$ stabilises each part of $\Sigma$; that is, $g \in N$. Moreover, the unique point contained in $A \cap \sigma_2$ is fixed by $g$, so $g=1$ since $N$ is semiregular. Thus $G_A=1$. Said otherwise, $G_{(A,\overline{A})}=1$, so $D(G)=2$. This proves part~(ii).
\end{proof}

\section{Computational methods}\label{s:computation}

We use this section to briefly discuss the computational methods that play a role in the proofs of Theorems~\ref{t:qp} and~\ref{t:sp}. For a permutation group $G$ of small degree, it is straightforward to determine whether $D(G)=2$ by carrying out a randomised computation in \textsc{Magma}. For example, in this way, we establish the following result using the Database of Transitive Groups \cite{cannonholt,hulpke}.

\begin{lemma}\label{l:computation} 
If $G$ is an imprimitive semiprimitive group of degree at most 47, then $D(G)=2$ or $G$ is $\GL(2,3)$ in its degree 8 action and $D(G)=3$. (That is, Theorems~\ref{t:qp} and~\ref{t:sp} are true for groups with degree at most 47.)
\end{lemma}

In Tables~\ref{tab:affP} and~\ref{tab:asP} we list the groups in $\mathcal{P}$ from Theorem~\ref{t:seress} (that is, the primitive groups with distinguishing number greater than 2 that do not contain the alternating group), together with the ID of each group in the current database of primitive groups \cite{colva} in \textsc{Magma} \cite{magma}. If a group $X$ has ID $[n,i]$, then $X$ is the $i$th group of degree $n$ in this database; that is, the command \texttt{PrimitiveGroup(n,i)} in \textsc{Magma} returns $X$. (We have listed the IDs because this database is different to the one used in \cite{seress}.)

\begin{table} 
\centering
\caption{The affine groups $X \leqslant \sym(n)$ in the set $\mathcal{P}$ } \label{tab:affP}
\begin{center} Groups with $D(X)=3$ \end{center}
\begin{tabular}{c|ccccc}
\toprule[1pt]   
ID  & $[5,2]$       & $[5,3]$        & $[7,4]$        & $[8,2]$            & $[9,4]$       \\
$X$	& $D_{10}$      & $F_{20}$       & $F_{21}$       & $2^3.F_{21}$       & $3^2.D_8$     \\
\hline
    & $[9,5]$       & $[9,6]$        & $[9,7]$        & $[16,16]$          & $[16,17]$     \\
    & $3^2.8.2$	    & $3^2.\SL(2,3)$ & $3^2.\GL(2,3)$ & $2^4.\GammaL(2,4)$ & $2^4.\alt(6)$ \\ 
\hline
    & $[16,18]$     & $[16,19]$      &  $[16,20]$     & $[32,3]$                           \\
    & $2^4.\sym(6)$ & $2^4.\alt(7)$  & $2^4.\alt(8)$  & $2^5.\GL(5,2)$                     \\
\bottomrule[1pt]
\end{tabular}
\bigskip
\begin{center} Groups with $D(X)=4$ \end{center}
\begin{tabular}{c|c}
\toprule[1pt]
ID  & $[8,3]$        \\
$X$	& $2^3.\GL(3,2)$ \\
\bottomrule[1pt]
\end{tabular}
\end{table}

\begin{table}
\centering
\caption{The almost simple groups $X \leqslant \sym(n)$ in the set $\mathcal{P}$} \label{tab:asP}
\begin{center} Groups with $D(X)=3$ \end{center}
\begin{tabular}{c|ccccc}
\toprule[1pt]
ID  & $[6,1]$           & $[8,4]$             &  $[8,5]$          & $[9,8]$           & $[9,9]$         \\
$X$	& $\PSL(2,5)$       & $\PSL(2,7)$         & $\PGL(2,7)$       & $\PSL(2,8)$       & $\PGammaL(2,8)$ \\
\hline 
    & $[10,2]$          & $[10,3]$            & $[10,4]$          & $[10,5]$          & $[10,6]$        \\
    & $\sym(5)$         & $\alt(6)$           & $\sym(6)$         & $\alt(6).2_2$     & $\alt(6).2_3$   \\
\hline
    & $[10,7]$          & $[11,5]$            & $[12,2]$          & $[12,3]$          & $[13,7]$        \\
    & $\alt(6).2^2$     & $\PSL(2,11)$        & $\PGL(2,11)$      & $\mathrm{M}_{11}$ & $\PSL(3,3)$     \\ 
\hline  
    & $[14,2]$          & $[15,4]$            & $[17,7]$          & $[17,8]$          & $[21,7]$        \\
    & $\PGL(2,13)$      & $\alt(8)$           & $\PSL(2,16).2$    & $\PGammaL(2,16)$  & $\PGammaL(3,4)$ \\ 
\hline 
    & $[22,1]$          & $[22,2]$            & $[23,5]$          & $[24,3]$                            \\
    & $\mathrm{M}_{22}$ & $\mathrm{M}_{22}.2$ & $\mathrm{M}_{23}$ & $\mathrm{M}_{24}$                   \\
\bottomrule[1pt]
\end{tabular}
\bigskip
\begin{center} Groups with $D(X)=4$.\end{center}
\begin{tabular}{c|cccc}
\toprule[1pt]
ID  & $[6,2]$     & $[7,5]$     & $[11,6]$          & $[12,4]$          \\
$X$	& $\PGL(2,5)$ & $\PSL(3,2)$ & $\mathrm{M}_{11}$ & $\mathrm{M}_{12}$ \\
\bottomrule[1pt]
\end{tabular}
\end{table}

\section{Proofs of the main results}\label{s:proof}
In this final section we prove Theorems~\ref{t:qp} and~\ref{t:sp}, together with Corollary~\ref{c:graph}.

\begin{proof}[Proof of Theorem \ref{t:qp}]
Let $G \leq \sym(\Omega)$ be quasiprimitive but not primitive. Since $G$ is not primitive, we may fix a maximal $G$-invariant partition $\Sigma$ of $\Omega$, and since $G$ is quasiprimitive, $G$ acts faithfully and primitively on $\Sigma$. Therefore, Theorem~\ref{t:seress} ensures that one of the following holds: $D(G^\Sigma) = 2$; $G^\Sigma \in \{\alt(\Sigma), \sym(\Sigma)\}$ or $G^\Sigma \in \mathcal{P}$.

First assume that $D(G^\Sigma) = 2$. Since $G$ is quasiprimitive, Lemma~\ref{l:reduction}(i) implies that $D(G) \leq D(G^\Sigma)=2$, which proves Theorem~\ref{t:qp} in this case.

For the remaining cases, we will assume that $D(G^\Sigma) > 2$. Let us fix a point $\omega \in \Omega$ and the part $\sigma \in \Sigma$ such that $\omega \in \sigma$. Write $n=|\Sigma|$, $H= G_\omega$ and $M = (G^\Sigma)_\sigma$. Note that $M$ is a maximal subgroup of $G$ since $\Sigma$ is a maximal $G$-invariant partition of $\Omega$, and note that $H < M$ since $G$ is not primitive. If $|M:H| \geq n-1$, then $|\sigma| \geq |\Sigma|-1$, so $D(G)=2$, by Lemma~\ref{l:reduction}(ii). Now we assume that $|M:H| < n-1$.

Assume that $G^\Sigma$ is $\alt(\Sigma)$ or $\sym(\Sigma)$. Therefore, $M \cong \alt(n-1)$ with $n\geq 4$ (if $G^\Sigma=\alt(\Sigma)$) or $M \cong \sym(n-1)$ with $n\geq 3$ (if $G^\Sigma = \sym(\Sigma)$).

Suppose that $n=3$. Then $G \cong \sym(3)$, $M \cong \sym(2)$ and $H=1$. However, $\alt(3)$ is a normal subgroup of $G$ which is intransitive on $\Omega$, which contradicts the fact that $G$ acts quasiprimitively on $\Omega$. Now suppose that $n=4$, so $G \cong \sym(4)$ or $G \cong \alt(4)$, and $|\Omega|\geq 8$. Since $V_4$ is a normal subgroup of $G$ which has order $4$, it is intransitive on $\Omega$, contradicting the fact that $G$ acts quasiprimitively on $\Omega$. Therefore, $n \geq 5$.

Suppose first that $G\cong \sym(n)$, $M \cong \sym(n-1)$ and $H \cong \alt(n-1)$. It follows that $H$ is contained in $\alt(n)$, so that $\alt(n)$ is an intransitive normal subgroup of $G$, a contradiction.

Suppose that $n \geq 6$. Then every proper subgroup of $\alt(n-1)$ has index at least $n-1$, and the only proper subgroup of $\sym(n-1)$ with index strictly less than $n-1$ is $\alt(n-1)$. Therefore, $G\cong \sym(n)$, $M \cong \sym(n-1)$ and $H \cong \alt(n-1)$; however, we have shown that this is not the case. We conclude that $n=5$.

If $G\cong \sym(5)$ and $M \cong \sym(4)$, then because of the cases ruled out above, we must have $H \cong D_8$. Similarly, if $G\cong \alt(5)$ and $M \cong \alt(4)$, then $H\cong V_4$. In both of these cases, the action of $G$ on the cosets of $H$ is isomorphic to the action of $G$ on the partitions of $\{1,2,3,4,5\}$ with parts of size $1,2,2$. It is straightforward to verify that there is a subset of size 3 (for example, $\{ 1|23|45, \, 1|24|35, \, 2|13|45 \}$ in terms of partitions) whose stabiliser in $G$ is trivial. Therefore, $D(G) = 2$, which proves Theorem~\ref{t:qp} in this case.

Finally, assume that $G^\Sigma \in \mathcal{P}$. Suppose first that $G^\Sigma$ is affine. Then there is an elementary abelian normal subgroup $V \leqn G^\Sigma$ of order $|\Sigma|=n$.  Since $G$ and $G^\Sigma$ are isomorphic as abstract groups, we have $V \leqn G$. Note that $|\Omega|=|\sigma|n\geq 2n$, so $V$ acts intransitively on $\Omega$, a contradiction to $G$ being quasiprimitive. Therefore, $G^\Sigma$ is not affine. 

It remains to assume that $G^\Sigma$ is almost simple, so $X=G^\Sigma$ appears in Table~\ref{tab:asP}. Recall that $M$ is a point stabiliser in the degree $n$ action of $G$ on $\Sigma$ and that we assumed that \mbox{$|M:H|<n-1$}. The subgroup $H$ is contained in at least one maximal subgroup $K$ of $M$, which must satisfy  $|M:K| < n-1$. For $(G,n)\in \{(\mathrm{M}_{11},12)$, $(\mathrm{M}_{12},12)$, $(\mathrm{M}_{22},22)$, $(\mathrm{M}_{23},23)$, $(\mathrm{M}_{24},24)\}$, every maximal subgroup $K$ of $M$ satisfies $|M:K|\geq n-1$, a contradiction. Assume $(G,n)\in\{(\sym(5),6),(\alt(6).2_2,10) ,(\mathrm{M}_{22}.2,22)\}$. Then the only maximal subgroups $K$ of $M$ satisfying $|M:K| < n-1$ also satisfy   $K\leq \mathrm{soc}(G)$. Hence $H\leq \mathrm{soc}(G)$ and the proper normal subgroup $\mathrm{soc}(G)$ of $G$ cannot be transitive, a contradiction to  $G$ being quasiprimitive.

From now on, we assume $G$ is one of the other groups in the list $\mathcal{P}$. Let $K$ be a maximal subgroup of $M$ containing $H$. By Lemma~\ref{l:reduction}(i), $D(G) \leq D(G^{\Sigma_K})$. Therefore, to prove that $D(G)=2$, it suffices to prove that $D(G^{\Sigma_K})=2$ for each maximal subgroup $K$ of $M$ with $|M:K| < n-1$. This is a straightforward computation in \textsc{Magma}. (Computer verification would be more difficult for some of the groups covered in the previous paragraphs and it is for this reason that we treat those groups separately.) This completes the proof.
\end{proof}

We now turn to semiprimitive groups. Before proving Theorem~\ref{t:sp}, we handle one case of the proof.
\begin{lemma}\label{l:sp_case}
Let $G \leq \sym(\Omega)$ be semiprimitive, let $N \leqn G$ be an intransitive normal subgroup of $G$ of prime order and let $\Sigma$ be the set of $N$-orbits. Assume that $G^\Sigma$ is an almost simple group in $\mathcal{P}$. Then $D(G)=2$.
\end{lemma}

\begin{proof}
Let $\sigma \in \Sigma$ and note that $|\sigma| = |N|$ since $N$ is semiregular. In particular, $|\Omega| = |N||\Sigma|$. Write $|N|=p$. By Lemma~\ref{l:sp}(i), $G^\Sigma\cong G/N$. If $|N| \geq |\Sigma| - 1$, then $D(G)=2$, by Lemma~\ref{l:reduction}(ii). Thus, we will assume that $|N| < |\Sigma| -1$.  Consequently, for each almost simple group $G^\Sigma \in \mathcal{P}$ we can determine all of the possibilities for $|N|$. 

Let $H=G^\infty$ be the last term in the derived series of $G$. Then $T \coloneqq HN/N = (G/N)^\infty$ is the socle of $G/N$, and thus is a non-abelian simple group. Since $(HN)^\Sigma$ is a normal subgroup of $G^\Sigma$ that is not semiregular (as $|HN^\Sigma| = |T| > |\Sigma|$) and since  $G^\Sigma$ is semiprimitive by Lemma \ref{l:sp}(ii), we know that $(HN)^\Sigma$ acts transitively on $\Sigma$. This implies that $HN$ acts transitively on $\Omega$ since $N$ is transitive on each of its orbits. Since $N$ has prime order, either $N \cap H = 1$ or $N \leq H$.  

First assume that $N \cap H = 1$. Then $H \cong HN/N = T$. Suppose further that $H$ is intransitive on $\Omega$. Since $G$ is semiprimitive, $H$ is semiregular. Let $\Delta$ be the set of $H$-orbits. By Lemma~\ref{l:sp}, $G^\Delta = G/H$ acts faithfully and semiprimitively on $\Delta$. Moreover, since $HN$ is transitive on $\Omega$, $N^\Delta = (HN)^\Delta = HN/H \cong N$ is transitive on $\Delta$. Noting that $N$ has prime order, we see that $N$ acts regularly on $\Delta$. In particular, $|\Delta|=|N|$ is prime, so $G^\Delta$ is primitive with a regular normal subgroup $N^\Delta$ of prime order $p$. Therefore, $G^\Delta$ is a subgroup of $\mathrm{AGL}(1,p)$. Consequently, for a point $\delta \in \Delta$, the stabiliser $(G^\Delta)_\delta$ is cyclic. By Lemma~\ref{l:sp}(iv), $(G^\Sigma)_\sigma \cong G_\omega \cong (G^\Delta)_\delta$, so $(G^\Sigma)_\sigma$ is cyclic. However, consulting Table~\ref{tab:asP}, we see that $(G^\Sigma)_\sigma$ is non-abelian, which is a contradiction. Therefore, $H$ acts transitively on $\Omega$. 

Consequently, $|H:H_\omega| = |\Omega|=|N||\Sigma| = p|\Sigma|$. Recall that $H\cong H^\Sigma= T$  and $p$ is a prime number such that $p < |\Sigma|-1$. Therefore, in \textsc{Magma} we can construct all possible degree $p|\Sigma|$ actions of $H$. Note $G \leq N_{\sym(p|\Sigma|)}(H)$ and $G$ has a subgroup $N$ of order $p$ centralising $H$. For each possible action, it is straightforward to verify in \textsc{Magma} whether $N_{\sym(p|\Sigma|)}(H)$ has a subgroup of order $p$ centralising $H$ and, if so, to then verify that $D(N_{\sym(p|\Sigma|)}(H))=2$. Since $G \leq N_{\sym(p|\Sigma|)}(H)$, we deduce that $D(G) = 2$.

We many now assume that $N \leq H$. Since $N$ is abelian $N \leq C_H(N)$. Since $H/N=T$ is simple, either $C_H(N)=N$ or $C_H(N)=H$. Suppose that $C_H(N)=N$. Then $T = H/N = H/C_H(N)$ embeds in $\mathrm{Aut}(N)$, which is a contradiction since $\mathrm{Aut}(N)$ is abelian. Therefore, $C_H(N)=H$, so $N \leq Z(H)$. In particular, $H$ is a perfect central extension of $T$ by $N$. Consequently, $H$ is a quotient of the universal covering group of $T$, and therefore $p=|N|$ divides the order of the Schur multiplier of $T$. Above we noted that $HN=H$ is transitive on $\Omega$, so we have $|H: H_\omega| = |\Omega| = p|\Sigma|$. As a result, $H$ must have a core-free subgroup of index $p|\Sigma|$. Considering each of the simple groups $T$ which occur as a socle of a group in Table~\ref{tab:asP}, this implies that
\[
(H,|\Omega|) \in \{ (\SL(2,7), 16), (\nonsplit{2}{\mathrm{M}_{12}},24), (\SL(3,4),63) \}.
\]
These actions are available at \cite{webatlas} and we verify that $D(N_{\sym(p|\Sigma|)}(H))=2$ with \textsc{Magma} in each case. Therefore, $D(G)=2$, completing the proof.
\end{proof}

We now prove Theorem~\ref{t:sp}.
\begin{proof}[Proof of Theorem~\ref{t:sp}]
By Lemma~\ref{l:computation}, the result holds for groups with degree at most $47$. In particular, $\GL(2,3)$ in its degree 8 action (on the set $\Omega$ of non-zero vectors of $\mathbb{F}_3^2$) has distinguishing number 3 (which is witnessed by the partition $\{ \{e_1\}, \{e_2\}, \Omega \setminus \{e_1,e_2\}\}$).

Now let $G$ be a counterexample to the statement of minimal order. Therefore, $G$ is semiprimitive but not quasiprimitive, $G$ is not $\GL(2,3)$ in its degree 8 action and $D(G)>2$. As noted above, this implies that $G$ has degree at least $48$. Since $G$ is not quasiprimitive, we may fix an intransitive minimal normal subgroup $N$ of $G$ (since otherwise all normal subgroups of $G$ would be transitive). Since $N$ is a minimal normal subgroup, we may write $N=T^d$ for a (possibly abelian) simple group $T$. Let $\Sigma$ be the set of $N$-orbits, write $n=|\Sigma|$ and record that $|\sigma|=|N|$ for all $\sigma \in \Sigma$. By Lemma~\ref{l:sp}, $G^\Sigma \cong G/N$ is semiprimitive. By Lemma~\ref{l:reduction}(i),  $D(G^\Sigma) \geq D(G) > 2$, so Lemma~\ref{l:reduction}(ii) implies that $|N|<n-1$. 

Suppose that $G^\Sigma$ is not quasiprimitive. Record that $D(G^\Sigma) > 2$ and $|G^\Sigma|=|G/N| < |G|$. Since $G$ is a minimal counterexample to the statement, we must have $G^\Sigma = \GL(2,3)$ and $n=8$. Since $N=T^d$ and $|N| < 7$, we may conclude that $|N|$ is a prime power and  $|N| \in \{2,3,4,5\}$. Therefore, the degree of $G$ is $8|N| \in \{16,24,32,40\}$. This is a contradiction since we know the degree of $G$ is at least $48$. Therefore, $G^\Sigma$ is quasiprimitive. Accordingly, Theorem~\ref{t:qp} implies that $G^\Sigma$ is primitive. Moreover, by Theorem~\ref{t:seress}, either $G^\Sigma \in \{\alt(n), \sym(n)\}$ or $G^\Sigma \in \mathcal{P}$.

Suppose that $G^\Sigma \in \{\alt(n), \sym(n)\}$. Recall that $N=T^d$ for a simple group $T$. Since $1 < |N| < n-1$ and $|\Omega|=|N|n\geq 48$, it follows that $n \geq 8$. Therefore, we are in a position to apply \cite[Theorem~1.7]{kyle}, and one of the following holds: $d \geq n-2$; $\alt(n-1)$ embeds in $T$; or $n=8$ and $|N|=16$. In the first case, $|N| = |T|^d \geq 2^{n-2} > n-2$, which is a contradiction. In the second case, $|N| \geq (n-1)! > n-1$, which again is a contradiction. In the third case, $|N|=16 > 8 = n$, a final contradiction. Therefore, $G^\Sigma \in \mathcal{P}$.

Since $G^\Sigma \in \mathcal{P}$, we know that $n \leq 32$ (see Tables~\ref{tab:affP} and~\ref{tab:asP}). Since  $N=T^d$ for a simple group $T$ and $|N| < n-1 \leq 31$, it follows that $N$ is an elementary abelian group of order $p^d$ for some prime $p$. Moreover, $N$ is an irreducible (but not necessarily faithful) module for $G/N$ over $\mathbb{F}_p$. Note that each member of $\mathcal{P}$ is either affine or almost simple. We will consider these two cases separately.

First assume that $G^\Sigma$ is affine. In this case, $n=r^f$ for a prime number $r$ and a positive integer $f$, and $G^{\Sigma} = G/N$ has an elementary abelian regular normal subgroup $R/N$. In particular, $R$ is a regular normal subgroup of $G$. Now $G/C_G(N)$ acts on $N$ by conjugation and $N \leq C_G(N)$ since $N$ is abelian. Since $R/N$ is a minimal normal subgroup of $G/N$, either $R/N \cap C_G(N)/N = 1$ or $R/N \leq C_G(N)/N$. In the first case, $R/N \cong C_r^f$ embeds into $\mathrm{Aut}(N) \cong \GL(d,p)$, but since $p^d < n-1$, by considering the possibilities for $(p,d)$ for each of the groups in Table~\ref{tab:affP} it is straightforward to see that this is impossible. For example, if $R/N=C_2^4$, then $p^d=|N|<n-1=15$ so either $d=1$ or $(p,d) \in \{(2,2),(2,3),(3,2)\}$. If $d=1$, then clearly $\GL(1,p) \cong C_{p-1}$ does not have a subgroup of order $2^4 > p-1$, and it is not difficult to verify that $C_2^4$ is not a subgroup in the three cases where $d > 1$ either. Therefore, $R/N \leq C_G(N)/N$, which implies that $R \leq C_G(N)$ and hence that $N\leq Z(R)$. 

Suppose that $|R/N|$ and $|N|$ are coprime. Then, by the Schur--Zassenhaus Theorem, $N$ has a complement $M$ in $R$. In particular, $|R|=|M||N|$. Moreover, since  $N\leq Z(R)$, $M$ is normal in $R$. Therefore, $M$ is a normal Sylow $r$-subgroup of $R$ and hence is characteristic in $R$. It follows that $M$ is a normal subgroup of $G$.  Now let $\Pi$ be the set of orbits of $M$ and let $\pi \in \Pi$. Note that $|\Omega|=|\Pi||M|$ and $|\Omega|=|R|$ since $R$ is regular. Therefore, $|N|=|\Pi|$. Earlier we observed that $|N| < |\Sigma|-1 = |\Omega|/|N|-1 = |M|-1$, so $|\Pi| < |M|-1$. Now Lemma~\ref{l:reduction}(ii) implies that $D(G) = 2$, which is a contradiction. Therefore, $|R/N|$ and $|N|$ are not coprime, that is $r=p$. 

Since $p^d < r^f-1$ and $p=r$ we may conclude that $d<f$. In particular $f>1$, so, by inspecting Table~\ref{tab:affP}, $r^f \in \{8,9,16,32\}$. Since the degree of $G$ is $|R| = r^{d+f}\geq 48$ with $d < f$, we must have that $r^f \in \{16,32\}$. Now since $N$ is in the centre of $R$, $G/R \cong (G^\Sigma)_\sigma$ acts irreducibly (but not necessarily faithfully) on $N$. Since $|N| = r^d < r^f$, considering Table~\ref{tab:affP}, we must have $d=1$ (since $f$ is the smallest dimension of a non-trivial irreducible representation of $(G^\Sigma)_\sigma$ in each case). Said otherwise, $|N|=r$, so the degree of $G$ is $|R|=r^{f+1}$. Therefore, $r^{f+1} \geq 48$ and since $r^f \in \{16,32\}$ we deduce that $r^f=32$, $G^\Sigma = 2^5.\mathrm{GL}(5,2)=\mathrm{AGL}(5,2)$ and $|N|=2$. Clearly $G$ must act trivially by conjugation on $N$, and thus $G$ is a central extension of $\mathrm{AGL}(5,2)$. Now $\mathrm{AGL}(5,2)$ is perfect and a {\sc Magma} calculation shows that $\mathrm H^2(\mathrm{AGL}(5,2),C_2) = 1$. This implies that $G$ contains a normal subgroup $H \cong \mathrm{AGL}(5,2)$ such that $G = HN$ and $[H,N]=1$ (that is, $G \cong H \times N$). Therefore, $[G,G] = [HN,HN] = [H,H][N,N] = [H,H] = H$ since $H \cong \mathrm{AGL}(5,2)$ is perfect. Now Lemma~\ref{l:sp}(iv) implies that $G_\omega \cong (G^\Sigma)_\sigma \cong \mathrm{GL}(5,2)$, which is also perfect. Therefore, $G_\omega \leq [G,G] = H$. Consequently, $HG_\omega = H < G$, so $H$ is intransitive, and $1 \neq G_\omega \leq H_\omega$, so $H$ is not semiregular. This is a contradiction to $G$ being semiprimitive. Therefore, $G^\Sigma$ is not affine.

Now assume that $G^\Sigma$ is almost simple. Let $H = G^\infty$ and write $T=HN/N=(G^\infty)^\Sigma$, noting that $T$ is the socle of $G^\Sigma$ (so, in particular, is a non-abelian simple group). Since $D(G) > 2$, Lemma~\ref{l:sp_case} implies that $d>1$, where $|N|=p^d$. As $N$ is abelian, $N \leq C_G(N)$. Since $HN/N$ is the minimal normal subgroup of $G/N$, either $HN/N \cap C_G(N)/N = 1$ or $HN/N \leq C_G(N)/N$. In the first case, $T \cong HN/N$ embeds into $\mathrm{Aut}(N) \cong \GL(d,p)$, but since $p^d < n-1$, by considering the possibilities for $(p,d)$ for each of the groups in Table~\ref{tab:asP} we see that this is impossible. For example, if $n=21$ and $T=\PSL(3,4)$, then $p^d=|N|<n-1=20$ with $d>1$ so $(p,d) \in \{(2,2),(2,3),(2,4),(3,2)\}$. If $(d,p) \neq (2,4)$, then $|\GL(d,p)| < |\PSL(3,4)|$ and the claim is immediate, while $|\GL(4,2)|=|\PSL(3,4)|$ but $\GL(4,2) \not\cong \PSL(3,4)$ so the claim is true in this case also. Therefore, $HN/N \leq C_G(N)/N$, which implies that $H \leq C_G(N)$.

Observe that $G/HN \cong (G/N)/(HN/N) = G^\Sigma/\mathrm{soc}(G^\Sigma)$ (which is easily determined for each group $G^\Sigma$ in Table~\ref{tab:asP}). Since $H \leq C_G(N)$, $N$ is an irreducible (but not necessarily faithful) module for $G/HN$. Suppose that $G/HN$ is a $p$-group. Then $G/HN$ has a fixed point on $N$ and since $N$ is irreducible this implies that $|N|=p$, which contradicts the fact that $d > 1$. Therefore, $G/HN$ is not a $p$-group. Now $G/HN \not\in \{1, C_2, C_2^2 \}$ since all irreducible modules over a field of odd order for $1$, $C_2$ and $C_2^2$ are one-dimensional. By inspecting the possibilities in Table~\ref{tab:asP}, we see that $(T,\, n,\,  G/HN,\,  N)$ must be one of
$(\PSL(2,8),\, 9,\, C_3,\, C_2^2),$    $( \PSL(2,16),\, 17,\, C_4,\, C_3^2)$, $( \PSL(3,4),\, 21,\, \sym(3),\, C_2^2)$.
Since $|\Omega|=n|N|\geq 48$ we may discount the first of these possibilities. We now consider the remaining two options.

First assume that $T = \PSL(2,16)$ and $N=C_3^2$. Since $N \leqslant Z(HN)$, $HN$ is a central extension of $N$ by $HN/N \cong T$. Since the Schur multiplier of $T$ is trivial, $HN \cong N \times T \cong C_3^2 \times \mathrm{PSL}(2,16)$ (see \cite[Section~33]{aschbacher}, for example). Since $H=G^{\infty}$, $H=[HN,HN] \cong T$, so  $H \cap N=1$. If $H$ is intransitive, then, since $G$ is semiprimitive, $H$ is semiregular, so $|\PSL(2,16)|=|H|$ divides $|\Omega| = n|N| =17 \cdot 3^2 = 153$, which is false. Therefore, $H$ is transitive and $H_\omega$ is a subgroup of index $|\Omega| = 153$. However, $H \cong \PSL(2,16)$ has no subgroup of index $153$, which is a contradiction.

Next assume that $T =\PSL(3,4)$ and $N=C_2^2$. Again $N \leqslant Z(HN)$, so $HN$ is a central extension of $N$ by $T$. Since $N$ is a minimal normal subgroup of $G$, either $H \cap N = 1$ or $N \leq H$. If $H \cap N = 1$, then $H \cong HN/N = \PSL(3,4)$. If $N \leq H$, then $N \leq Z(H)$ and $H$ is a perfect central extension of $\PSL(3,4)$ by $C_2^2$; since the Schur multiplier of $\PSL(3,4)$ is $C_3 \times C_4^2$, there is a unique such extension, which is described explicitly in \cite[p.~245]{holtplesken} and which we can construct in \textsc{Magma}. In both cases $|H|\geq |\PSL(3,4)|=20160$. Since $H \leqn G$, $H$ is either semiregular or transitive. If $H$ is semiregular, then $|H|$ divides $|\Omega| = 21 \cdot 4 = 84$, which is absurd. If $H$ is transitive, then $H$ has a subgroup of index $84$ and for both possibilities for $H$ described above we can verify in \textsc{Magma} that this is false. This final contradiction completes the proof.
\end{proof}

We conclude by proving Corollary~\ref{c:graph}.

\begin{proof}[Proof of Corollary~\ref{c:graph}]
Let $\Omega$ be the vertex set of $\Gamma$ and write $G=\mathrm{Aut}(\Gamma)$. Assume that $|\Omega| \geq 4$; the result is clear otherwise. Suppose that $D(G) > 2$. Then Corollary~\ref{c:corollary} implies that $\alt(\Omega) \leq G$, $G \in \mathcal{P}$ or $G$ is $\GL(2,3)$ of degree 8. If $G$ is  $2$-transitive, then $\Gamma$ is  complete. Hence $\alt(\Omega) \not\leq G$ and if $G\in \mathcal P$ then $G$ is $D_{10}$ of degree $5$, $3^2.D_8$ of degree $9$, or $\sym(5)$ of degree $10$. These permutation groups have rank $3$, so are the automorphism groups of at most one graph and its complement; these are the graphs mentioned in the statement. If $G=\GL(2,3)$, then $G$ also has rank 3 and the orbitals of $G$ imply that $\Gamma$ is the complete multipartite graph with 4 parts each with two vertices. However, $\mathrm{Aut}(\Gamma)=\sym(2)\wr\sym(4) \neq G$, a contradiction. Therefore, $\Gamma$ is listed in the statement, and we verify that $D(\mathrm{Aut}(\Gamma))=3$ in each case.
\end{proof}

\end{document}